\documentclass[11pt]{amsart}
\usepackage{graphicx}
\usepackage{multicol,multirow}
\usepackage{amsmath,amssymb,amsfonts}
\usepackage{mathrsfs}
\usepackage{amsthm}
\usepackage{rotating}
\usepackage{appendix}
\usepackage[numbers]{natbib}
\usepackage{array}
\usepackage{hyperref}

\newtheorem{theorem}{Theorem}[section]
\newtheorem{lemma}[theorem]{Lemma}
\newtheorem{definition}{Definition}[section]
\newtheorem*{thm*}{Theorem}
\newtheorem{remark}[theorem]{Remark}
\newtheorem{example}[theorem]{Example}

\usepackage{color}
\usepackage{graphicx}
\usepackage{color}
\usepackage{graphicx}
\usepackage{etoolbox}
\usepackage{tikz}
\usepackage{adjustbox}
\usepackage{wrapfig}
\usepackage{diagbox}
\usepackage{array}
\usepackage{multirow}
\usepackage{todonotes}

\newcommand{\hop}{\vskip .3cm\noindent} 
\newcommand{\hip}{\vskip .1cm\noindent}

\newcommand{\Isom}{\operatorname{Isom}}

\begin{document}


\title[\resizebox{4.6in}{!}{Minimal growth rates in higher dimensions}]{Hyperbolic Coxeter groups of minimal growth rates in higher dimensions}

\author{Naomi Bredon}
\address{Department of Mathematics\\
University of Fribourg\\
CH-1700 Fribourg\\Switzerland}
\email{naomi.bredon@unifr.ch}



\begin{abstract}
The cusped hyperbolic $n$-orbifolds of minimal volume are well known for $n\leq 9$. Their fundamental groups are related to the Coxeter $n$-simplex groups $\Gamma_n$ listed in Table \ref{Poly}. In this work, we prove that $\Gamma_n$ has minimal growth rate among all non-cocompact Coxeter groups of finite covolume in $\hbox{Isom}\mathbb H^n$. In this way, we extend previous results of Floyd for $n=2$ and of Kellerhals for $n=3$ respectively. Our proof is a generalisation of the methods developed in \cite{Bred-Kell} for  the cocompact case.
\end{abstract}

\maketitle

\section{Introduction}
\label{section1}

\noindent
Let $\mathbb H^n$ denote the real hyperbolic $n$-space with its isometry group $\hbox{Isom}\mathbb H^n$. 
\newline
A {\em hyperbolic Coxeter polyhedron} $P\subset \mathbb H^n$ is a convex polyhedron of finite volume all of whose dihedral angles are integral submultiples of $\pi$. Associated to $P$ is the {\em hyperbolic Coxeter group} $\Gamma \subset \hbox{Isom}\mathbb H^n$ generated by the reflections in the bounding hyperplanes of $P$. 
By construction, $\Gamma$ is a discrete group with associated orbifold $O^n=\mathbb H^n / \Gamma$ of finite volume.

\hip
We focus on {\em non-compact} hyperbolic Coxeter polyhedra, having at least one ideal vertex $v_{\infty} \in \partial \mathbb H^n$. Notice that the stabiliser of the vertex $v_{\infty}$ is an affine Coxeter group. The group $\Gamma$ is called {\em non-cocompact}, and its quotient space $O^n$ has at least one cusp.  
\hip 
The hyperbolic Coxeter group $\Gamma$ is the geometric realisation of an  abstract Coxeter system $(W,S)$ consisting of a group $W$ with a finite generating set $S$ together with the relations $s^2=1$ and $(ss')^{m_{ss'}}=1$, where $m_{ss'}=m_{s's}\in\{2,3,\ldots,\infty\}$ for all $s,s'\in S$ with $s\not=s'$. The {\em growth series} $f_S(t)$ of $W=(W,S)$  is given by 
\[f_S(t)=1+\sum\limits_{k\ge1}a_kt^k \, , \]
 where $a_k\in\mathbb Z$ is the number of words in $W$ with $S$-length $k$. The {\em growth rate} $\tau_{W}$ of $W=(W,S)$ is defined as the inverse of the radius of convergence of $f_S(t)$.
\hip 
We are interested in small growth rates of non-cocompact hyperbolic Coxeter groups in $\hbox{Isom}\mathbb H^n$ for $n\geq 2$. For $n=2$, Floyd \cite{Fl} showed that the Coxeter group $\Gamma_2 = [3,\infty]$ generated by the reflections in the triangle with angles $\pi/2, \pi/3$ and $0$ is the (unique) group of minimal growth rate. For $n=3$, Kellerhals \cite{K0} proved that the tetrahedral group $\Gamma_3$ generated by the reflections in the Coxeter tetrahedron with symbol $[6,3,3]$ realises minimal growth rate in a unique way.

\hip

Consider the hyperbolic Coxeter $n$-simplices and their reflection groups $\Gamma_n \subset \hbox{Isom}\mathbb H^n$ depicted in Table \ref{Poly}. 
For their volumes, we refer to \cite{JKRT}. Observe that $\Gamma_n$ is of minimal covolume among all hyperbolic Coxeter $n$-simplex groups. 

The aim of this work is to prove the following result in the context of growth rates.

\begin{thm*}[Theorem]\label{NB} Let $2\leq n \leq 9$. Among all non-cocompact hyperbolic Coxeter groups of finite covolume in $\Isom \mathbb H^n$, the group $\Gamma_n$ given in Table \ref{Poly} has minimal growth rate, and as such, it is unique.
\end{thm*}

\vspace{-3pt} 
\begin{table}[!h]
\tabcolsep=9pt%
\renewcommand*{\arraystretch}{3.5}
\begin{center}
\begin{tabular}{c c  c c }
$\Gamma_2$ & 
\begin{tikzpicture}
\tikzstyle{every node}=[font=\small]
\fill[black] (0,0) circle (0.06cm);
\fill[black] (1/2,0) circle (0.06cm);
\fill[black] (2/2,0) circle (0.06cm);
\draw (0,0) -- (1/2,0)  node [above,midway] {$\infty$} ;
\draw (1/2,0) -- (2/2,0);
\end{tikzpicture}  &
$\Gamma_3$ &
\begin{tikzpicture}
\tikzstyle{every node}=[font=\small]
\fill[black] (0,0) circle (0.06cm);
\fill[black] (1/2,0) circle (0.06cm);
\fill[black] (2/2,0) circle (0.06cm);
\fill[black] (3/2,0) circle (0.06cm);
\draw (0,0) -- (1/2,0) node [above,midway] {6} ; 
\draw (1/2,0) -- (2/2,0) ;
\draw (2/2,0) -- (3/2,0)   ;
\end{tikzpicture}  
\\ 
\hline
$\Gamma_4$ &
\begin{tikzpicture}
\tikzstyle{every node}=[font=\small]
\fill[black] (0,0) circle (0.06cm);
\fill[black] (1/2,0) circle (0.06cm);
\fill[black] (2/2,0) circle (0.06cm);
\fill[black] (3/2,0) circle (0.06cm);
\fill[black] (1/2,1/3) circle (0.06cm);
\draw (0,0) -- (1/2,0)  node [above,midway] {4} ;
\draw (1/2,0) -- (2/2,0) ;
\draw (2/2,0) -- (3/2,0)  ;
\draw (1/2,1/3) -- (1/2,0);
\end{tikzpicture}  
&
$\Gamma_5$ &
\begin{tikzpicture}
\tikzstyle{every node}=[font=\small]
\fill[black] (0,0) circle (0.06cm);
\fill[black] (1/2,0) circle (0.06cm);
\fill[black] (2/2,0) circle (0.06cm);
\fill[black] (3/2,0) circle (0.06cm);
\fill[black] (4/2,0) circle (0.06cm);
\fill[black] (5/2,0) circle (0.06cm);
\draw (0,0) -- (1/2,0);
\draw (1/2,0) -- (2/2,0) node [above,midway] {4};
\draw (2/2,0) -- (3/2,0) ;
\draw (4/2,0) -- (3/2,0) ;
\draw (4/2,0) -- (5/2,0) ;
\end{tikzpicture} 
\\
$\Gamma_6$ &
\begin{tikzpicture}
\tikzstyle{every node}=[font=\small]
\fill[black] (0,0) circle (0.06cm);
\fill[black] (1/2,0) circle (0.06cm);
\fill[black] (2/2,0) circle (0.06cm);
\fill[black] (3/2,0) circle (0.06cm);
\fill[black] (4/2,0) circle (0.06cm);
\fill[black] (5/2,0) circle (0.06cm);
\fill[black] (3/2,1/3) circle (0.06cm);
\draw (0,0) -- (1/2,0) node [above,midway] {4};
\draw (1/2,0) -- (2/2,0) ;
\draw (2/2,0) -- (3/2,0) ;
\draw (4/2,0) -- (3/2,0) ;
\draw (4/2,0) -- (5/2,0) ;
\draw (3/2,1/3) -- (3/2,0) ;
\end{tikzpicture} 
&
$\Gamma_7$&
\begin{tikzpicture}
\fill[black] (2,0) circle (0.06cm);
\fill[black] (2.5,0) circle (0.06cm);
\fill[black] (3,0) circle (0.06cm);
\fill[black] (3.5,0) circle (0.06cm);
\fill[black] (4,0) circle (0.06cm);
\fill[black] (4.5,0) circle (0.06cm);
\fill[black] (3,1/3) circle (0.06cm);
\fill[black] (3,2/3) circle (0.06cm);
\draw (2,0) -- (3,0) ;
\draw (4,0) -- (3,0) ;
\draw (4,0) -- (4.5,0) ;
\draw (3,1/3) -- (3,0) ;
\draw (3,1/3) -- (3,2/3) ;
\end{tikzpicture}\\
$\Gamma_8$ & 
\begin{tikzpicture}
\fill[black] (2.5,0) circle (0.06cm);
\fill[black] (3,0) circle (0.06cm);
\fill[black] (3.5,0) circle (0.06cm);
\fill[black] (4,0) circle (0.06cm);
\fill[black] (4.5,0) circle (0.06cm);
\fill[black] (5,0) circle (0.06cm);
\fill[black] (5.5,0) circle (0.06cm);
\fill[black] (6,0) circle (0.06cm);
\fill[black] (4,1/3) circle (0.06cm);
\draw (2.5,0) -- (3,0) ;
\draw (4,0) -- (3,0) ;
\draw (4,0) -- (6,0) ;
\draw (4,1/3) -- (4,0) ;
\end{tikzpicture}
&
$\Gamma_9$ & 
\begin{tikzpicture}
\fill[black] (1/2,0) circle (0.06cm);
\fill[black] (2/2,0) circle (0.06cm);
\fill[black] (3/2,0) circle (0.06cm);
\fill[black] (4/2,0) circle (0.06cm);
\fill[black] (5/2,0) circle (0.06cm);
\fill[black] (6/2,0) circle (0.06cm);
\fill[black] (7/2,0) circle (0.06cm);
\fill[black] (8/2,0) circle (0.06cm);
\fill[black] (9/2,0) circle (0.06cm);
\fill[black] (3/2,1/3) circle (0.06cm);
\draw (1/2,0) -- (3/2,0) ;
\draw (4/2,0) -- (3/2,0) ;
\draw (4/2,0) -- (5/2,0) ;
\draw (6/2,0) -- (5/2,0) ;
\draw (6/2,0) -- (7/2,0) ;
\draw (7/2,0) -- (8/2,0) ;
\draw (9/2,0) -- (8/2,0) ;
\draw (3/2,1/3) -- (3/2,0) ;
\end{tikzpicture}
\\
\end{tabular}
\end{center}
\caption{The hyperbolic Coxeter $n$-simplex group $\Gamma_n$}
\label{Poly}
\end{table}

Our Theorem should be compared with the volume minimality results for cusped hyperbolic $n$-orbifolds $O^n$ for $2\le n \le 9$. These results are due to Siegel \cite{Sie} for $n=2$, Meyerhoff \cite{Mey} for $n= 3$, Hild-Kellerhals \cite{HK} for $n=4$, and to Hild \cite{Hild} for $n\le 9$. Indeed, the fundamental group of $O^n$ is related to $\Gamma_n$ in all these cases.

\hip 

The work is organised as follows. In Section \ref{section2.1} 
we set the background about hyperbolic Coxeter polyhedra and their associated reflection groups. Furthermore, we present a result of Felikson and Tumarkin about their combinatorics as given by  Theorem  \cite[Theorem B]{FT1} which will be a play a crucial role in our proof.   In fact, we will exploit the (non-)simplicity of the Coxeter polyhedra in a most useful way. 
In Section \ref{section2.2}, we discuss growth series and growth rates of Coxeter groups and introduce the notion of extension of a Coxeter graph. We provide also some illustrating examples. 
The monotonicity result of Terragni \cite{Terragni}  for growth rates, presented in Theorem \ref{terr}, will be an another major ingredient in our proof.
Finally, Section \ref{section3} is devoted to the proof of our result. We perform it in two steps by assuming that the Coxeter graph under consideration has an affine component of type $\widetilde A_1$ or not.

\vspace{10pt}
{\em Acknowledgement.\quad}The author would like to express her gratitude to her supervisor Ruth Kellerhals for all the expert advice and support throughout this project.

\section{Hyperbolic Coxeter groups and growth rates}\label{section2}
\subsection{Coxeter polyhedra and their reflection groups}\label{section2.1}

Let $\mathbb X^n$ denote one of the standard geometric $n$-spaces, the unit $n$-sphere $\mathbb S^n$, the Euclidean $n$-space $\mathbb E^n$ or the real hyperbolic $n$-space $\mathbb H^n$.  As usual, we embed $\mathbb X^n$ in a suitable quadratic space $\mathbb Y^{n+1}$. In the Euclidean case, we take the affine model $\mathbb Y^{n+1}=\mathbb E^n \times \{0\}$.  In the hyperbolic case, we interpret $\mathbb H^n$ as the upper sheet of the hyperboloid in $\mathbb R^{n+1}$, that is, 
\[
\mathbb H^n=\{x\in\mathbb R^{n+1}\mid \,\langle x,x\rangle_{n,1}\,=-1\,,\,x_{n+1}>0\}\,,
\]
where $\,\langle x,x\rangle_{n,1}\,=x_1^2+\dots+x_n^2-x_{n+1}^2$ is the standard Lorentzian form. 
Its boundary $\partial \mathbb H^n$ can be identified with the set 
\[
\partial \mathbb H^n=\{x\in\mathbb R^{n+1}\mid \,\langle x,x\rangle_{n,1}\,=0\,,\,\sum_{k=1}^{n+1}x_k^2=1\,,\,x_{n+1}>0\}\,.
\]
In this picture, the isometry group of $\mathbb H^n$ is isomorphic to the group $PO(n,1)$ of positive Lorentzian matrices leaving the bilinear form $\langle \,, \rangle_{n,1}$ and the upper sheet invariant.
\hip
It is well known that each isometry of $\mathbb X^n$ is a finite composition of reflections in hyperplanes, where a hyperplane $H=H_v$ in $\mathbb X^n$ is characterised by a normal unit 
vector $v \in \mathbb Y^{n+1}$.  
Associated to $H_v$ are two closed half-spaces. We denote by $H^-_v$ the half-space in $\mathbb X^n$ with outer normal vector $v$. 

A (convex) {\em $n$-polyhedron} $P= \cap_{i\in I} H_{i}^- \subset \mathbb X^n$ is the non-empty intersection of a finite number of half-spaces $H_{i}^-$ bounded by the hyperplanes $H_{i}=H_{v_i}$ for $i\in I$.
A {\em facet} of $P$ is of the form $F_i=P\cap H_i$ for some $i\in I$. In the sequel, for $\mathbb X^n \neq \mathbb S^n$, we always assume that $P$ is of finite volume. In the Euclidean case, the implies that $P$ is compact, and in the hyperbolic case, $P$ is the convex hull of finitely many points $v_1,\dots,v_k \in \mathbb H^n \cup \partial\mathbb H^n$. If $v_i \in \mathbb H^n$, then $v_i$ is an {\em ordinary vertex}, and if $v_i\in \partial\mathbb H^n$, then $v_i$ is an {\em ideal vertex} of $P$, respectively. 

If all dihedral angles $\alpha_{ij}=\measuredangle(H_i,H_j)$ formed by intersecting hyperplanes $H_i, H_j$ in the boundary of $P$ are of the form $\frac{\pi}{m_{ij}}$ for an integer $m_{ij} \geq 2$, then $P$ is called a \emph{Coxeter polyhedron} in $\mathbb X^n$. Observe that the Gram matrix $\hbox{Gr}(P)=(\langle v_i, v_j\rangle_{\mathbb Y^{n+1}})_{i,j\in I}$ is a real symmetric matrix with $1$'s on the diagonal and non-positive coefficients off the diagonal. In this way, the theory of Perron-Frobenius applies.  For further details and references about Coxeter polyhedra in $\mathbb X^n$, we refer to \cite{F-web, V1, VinGII}.

\hip

Let $P= \cap_{i=1}^N H_i^{-}\subset \mathbb X^n$ 
be a Coxeter $n$-polyhedron. Denote by $r_i=r_{H_i}$ the reflection in the bounding hyperplane $H_i$ of $P$, and let $G=G_P$ be the group generated by $r_1, \ldots, r_N$.
It follows that $G$ is a discrete subgroup of finite covolume in $\hbox{Isom}\mathbb X^n$, called a {\em geometric Coxeter group}.

A geometric  Coxeter group $G\subset \hbox{Isom}\mathbb X^n$ with generating system $S=\{r_1,\ldots,r_N\}$ is the geometric realisation of an abstract Coxeter system $(W,S)$. In fact, we have $r_i^2=1$ and $(r_ir_j)^{m_{ij}}=1$ with $m_{ij}=m_{ji}\in\{2,3,\dots,\infty\}$ as above. Here, $m_{ij}=\infty$ indicates that $r_ir_j$ is of infinite order. 

For $\mathbb X^n= \mathbb S^n$, $G$ is a {\em spherical} Coxeter group and as such finite. For $\mathbb X^n= \mathbb E^n$, $G$ is a {\em Euclidean} or  {\em affine} Coxeter group and of infinite order. By a result of Coxeter \cite{Coxeter}, the irreducible spherical and Euclidean Coxeter groups are entirely classified. In contrast to this fact, {\em hyperbolic} Coxeter groups are far from being classified. For a survey about partial classification results, we refer to \cite{F-web}.

\hip

For the description of abstract and geometric Coxeter groups, one commonly uses the language of weighted graphs and Coxeter symbols.
Let $(W,S)$ be an abstract Coxeter system with generating system $S=\{s_1,\ldots,s_N\}$ and relations of the form $s_i^2=1$ and $s_is_j^{m_{ij}}=1$ with $m_{ij}=m_{ji}\in\{2,3,\dots,\infty\}$.
The {\em Coxeter graph} of the Coxeter system $(W,S)$ is the non-oriented graph $\Sigma$ whose nodes correspond to the generators $s_1,\ldots,s_N$.   If $s_i$ and $s_j$ do not commute, their nodes $n_i,n_j$ are connected by an edge with weight $m_{ij}\geq 3$.  We omit the weight $m_{ij}=3$ since it occurs frequently. The number $N$ of nodes is the {\em order} of $\Sigma$. 
A subgraph $\sigma\subset \Sigma$ corresponds to a {\em special} subgroup of $(W,S)$, that is, a subgroup of the form $(W_T,T)$ for a subset $T\subset S$. Observe that the Coxeter graph $\Sigma$ is connected if $(W,S)$ is irreducible. 

\hip 

In the case of a geometric Coxeter group $G=(W,S)\subset \hbox{Isom}\mathbb X^n$, we call its Coxeter graph $\Sigma$ {\em spherical}, {\em affine}, or {\em hyperbolic}, if $\mathbb X^n=\mathbb S^n, \mathbb E^n$ or $\mathbb H^n$, respectively.
In Table \ref{affinediag}, we reproduce all the connected affine Coxeter graphs, using the classical notation, with the exception of the three groups $\widetilde E_6, \widetilde E_7, \widetilde E_8$ (they will not appear in the following).

\begin{table}[!h]
\bigskip
\tabcolsep=10pt%
\renewcommand*{\arraystretch}{2.2}
\caption{Connected affine Coxeter graphs of order $n+1$}
\begin{center}
\begin{tabular}{| c c |  c c |}
\hline
 $\widetilde A_n$ & \begin{tikzpicture}
\fill[black] (0,0) circle (0.06cm);
\fill[black] (1/2,0) circle (0.06cm);
\fill[black] (1.6,0) circle (0.06cm);
\fill[black] (2.1,0) circle (0.06cm);
\fill[black] (1.1,1/3) circle (0.06cm);
\draw (0,0) -- (1/2,0) ;
\draw (0.7,0) -- (1/2,0);
\draw (0.7,0) -- (1.4,0) [dotted];
\draw (1.4,0) -- (1.6,0);
\draw (2.1,0) -- (1.6,0);
\draw (0,0) -- (1.1,1/3) ;
\draw (2.1,0) -- (1.1,1/3) ;
\end{tikzpicture} 
& $ \widetilde A_1$ & \begin{tikzpicture}
\tikzstyle{every node}=[font=\small]
\fill[black] (0,0) circle (0.06cm);
\fill[black] (1/2,0) circle (0.06cm);
\draw (0,0) -- (1/2,0) node [above,midway]{$\infty$} ;
\end{tikzpicture} \\
$\widetilde B_n$  & \begin{tikzpicture}
\tikzstyle{every node}=[font=\small]
\fill[black] (1/2,0) circle (0.06cm);
\fill[black] (1,0) circle (0.06cm);
\fill[black] (1.5,0) circle (0.06cm);
\fill[black] (2.5,0) circle (0.06cm);
\fill[black] (3,-1/4) circle (0.06cm);
\fill[black] (3,1/4) circle (0.06cm);
\draw (1,0) -- (1/2,0) node [above,midway]{4};
\draw (1,0) -- (1.5,0) ;
\draw (2.5,0) -- (3,1/4) ;
\draw (2.5,0) -- (3,-1/4) ;
\draw (1.5,0) -- (1.7,0) ;
\draw (2.3,0) -- (1.7,0)[dotted];
\draw (2.3,0) -- (2.5,0) ;
\end{tikzpicture}  & $ \widetilde G_2$ &
\begin{tikzpicture}
\tikzstyle{every node}=[font=\small]
\fill[black] (0,0) circle (0.06cm);
\fill[black] (1/2,0) circle (0.06cm);
\fill[black] (1,0) circle (0.06cm);
\draw (0,0) -- (1/2,0) node [above,midway]{$6$} ;
\draw (1,0) -- (1/2,0) ;
\end{tikzpicture}\\
 $\widetilde C_n$ & \begin{tikzpicture}
\tikzstyle{every node}=[font=\small]
\fill[black] (1/2,0) circle (0.06cm);
\fill[black] (1,0) circle (0.06cm);
\fill[black] (1.5,0) circle (0.06cm);
\fill[black] (2.5,0) circle (0.06cm);
\fill[black] (3,0) circle (0.06cm);
\draw (1,0) -- (1/2,0) node [above,midway]{4};
\draw (1,0) -- (1.5,0) ;
\draw (2.5,0) -- (3,0)  node [above,midway]{4};
\draw (1.5,0) -- (1.7,0) ;
\draw (2.3,0) -- (1.7,0)[dotted];
\draw (2.3,0) -- (2.5,0) ;
\end{tikzpicture}  & $ \widetilde F_4$  & \begin{tikzpicture}
\tikzstyle{every node}=[font=\small]
\fill[black] (0,0) circle (0.06cm);
\fill[black] (1/2,0) circle (0.06cm);
\fill[black] (2/2,0) circle (0.06cm);
\fill[black] (3/2,0) circle (0.06cm);
\fill[black] (4/2,0) circle (0.06cm);
\draw (0,0) -- (1/2,0)  ; 
\draw (2/2,0) -- (1/2,0) node [above, midway]{$4$}; 
\draw (2/2,0) -- (3/2,0) ;
\draw (3/2,0) -- (4/2,0) ; 
\end{tikzpicture}  \\
 $\widetilde D_n$ & \begin{tikzpicture}
\fill[black] (1/2,1/4) circle (0.06cm);
\fill[black] (1/2,-1/4) circle (0.06cm);
\fill[black] (1,0) circle (0.06cm);
\fill[black] (1.5,0) circle (0.06cm);
\fill[black] (2.5,0) circle (0.06cm);
\fill[black] (3,-1/4) circle (0.06cm);
\fill[black] (3,1/4) circle (0.06cm);
\draw (1/2,1/4) -- (1,0) ;
\draw (1/2,-1/4) -- (1,0) ;
\draw (1,0) -- (1.5,0) ;
\draw (2.5,0) -- (3,1/4) ;
\draw (2.5,0) -- (3,-1/4) ;
\draw (1.5,0) -- (1.7,0) ;
\draw (2.3,0) -- (1.7,0)[dotted];
\draw (2.3,0) -- (2.5,0) ;
\end{tikzpicture}  && \\
\hline
\end{tabular}
\end{center}
\label{affinediag}
\end{table}

An abstract Coxeter group with a simple presentation  can conveniently be  described by its {\em Coxeter symbol}. For example, the linear Coxeter graph with edges of successive weights $k_1,\dots,k_N\geq 3$ is abreviated by the Coxeter symbol $[k_1,\dots,k_N]$.  The Y-shaped graph made of one edge with weight $p$ and of two strings of $k$ and $l$ edges emanating from a central vertex of valency $3$ is denoted by $[p,3^{k,l}]$ (see \cite{JKRT}).

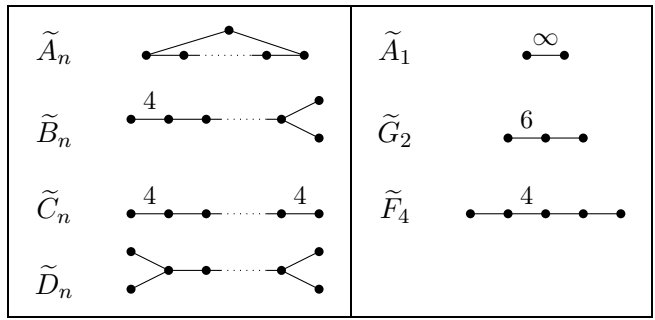
\hip 
\hip

Let us specify the context and consider a Coxeter polyhedron $P = \cap_{i=1}^N H_{i}^- $ in $\mathbb H^n$. Denote by $\Gamma = G_P \subset \hbox{Isom}\mathbb H^n$ its asssociated Coxeter group and by $\Sigma$ its Coxeter graph. Since $P$ is of finite volume, the graph $\Sigma$ is connected. Furthermore, if $P$ is not compact, then $P$ has at least one ideal vertex.

\hip
Let $v \in \mathbb H^n$ be an ordinary vertex of $P$. Then, its {\it link} $L_v$ is the intersection of $P$ with a small sphere of centre $v$ that does not intersect any facet of $P$ not incident to $v$. It corresponds to a spherical Coxeter polyhedron of $ \mathbb S^{n-1}$ and therefore to a spherical Coxeter subgraph $\sigma$ of order $n$ in $\Sigma$.  
\hip
Let $v_{\infty} \in \partial \mathbb H^n$  be an ideal vertex of $P$. Then, its link, denoted by $L_{\infty}$, is given by the intersection of $P$ with a sufficiently small horosphere centred at $v_{\infty}$ as above. The link $L_{\infty}$ corresponds to a Euclidean Coxeter polyhedron in $\mathbb E ^{n-1}$ and is related to an affine Coxeter subgraph  $\sigma_{\infty}$ of order $\geq n$ in $\Sigma$. 

\hip
More precisely, if $v_{\infty}$ is a {\em simple} ideal vertex, that is, $v_{\infty}$ is the intersection of exactly $n$ among the $N$ bounding hyperplanes of $P$, the Coxeter graph $\sigma_{\infty}$ is connected and of order $n$. Otherwise, $\sigma_{\infty}$ has $n_c(\sigma_{\infty}) \geq 2$ affine components, and we have the following formula.  \begin{equation}\label{rankorder}
n-1=\hbox{order}(\sigma_{\infty})-n_{c}(\sigma_{\infty}) \, .
\end{equation}
Recall that a polyhedron is {\em simple} if all of its vertices are simple.

\hip

As in the spherical and Euclidean cases, hyperbolic Coxeter 
simplices in $\mathbb H^n$ are all known, and they exist for $n\le 9$ (see \cite{Bou} or \cite{VinGII}). A list of their Coxeter graphs, Coxeter symbols and volumes can be found in \cite{JKRT}. Among the related Coxeter $n$-simplex groups, the group $\Gamma_n$, as given in Table \ref{Poly}, is of minimal covolume.

\hip
The following structural result for {\em simple} hyperbolic Coxeter polyhedra due to  Felikson and Tumarkin \cite[Theorem B]{FT1} will be a corner stone for the proof of our Theorem.  

\begin{theorem}\label{thmFT}
 Let  $n \leq 9$, and let $P\subset \mathbb H^n$ be a non-compact simple Coxeter polyhedron.
 If $P$ has mutually intersecting facets, then $P$ is either a simplex or isometric to the polyhedron  $P_0$ whose Coxeter graph is depicted in Figure \ref{fig:unique}. 
\end{theorem}

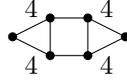
\begin{figure}[!h]
\centering
\begin{tikzpicture}
\tikzstyle{every node}=[font=\small]
\fill[black] (-1/2,1/4) circle (0.06cm);
\fill[black] (0,0) circle (0.06cm);
\fill[black] (1/2,0) circle (0.06cm);
\fill[black] (0,1/2) circle (0.06cm);
\fill[black] (1/2,1/2) circle (0.06cm);
\fill[black] (1,1/4) circle (0.06cm);
\draw (0,0) -- (-1/2,1/4) node [below,midway] {4} ;
\draw (-1/2,1/4) -- (0,1/2) node [above,midway] {4} ;
\draw (1/2,0) -- (1/2,1/2) ;
\draw (1/2,1/2) -- (0,1/2) ;
\draw (1/2,0) -- (0,0) ;
\draw (1,1/4) -- (1/2,0) node [below,midway] {4} ;
\draw (1,1/4) -- (1/2,1/2) node [above,midway] {4} ;
\draw (0,0) -- (0,1/2) ; 
\end{tikzpicture} 
\caption{The Coxeter polyhedron $P_0\subset\mathbb H^4$}
\label{fig:unique}
\end{figure}

\subsection{Growth rates and their monotonicity}\label{section2.2} 
Let $(W,S)$ be a Coxeter system and denote by $a_k\in\mathbb Z$ the number of words $w\in W$ with $S$-length $k$.  The {\em growth series} $f_S(t)$ of $(W,S)$ is defined by \[
f_S(t)=1+\sum\limits_{k\ge1}a_kt^k.\]
\hip
In the following, we list some properties of $f_S(t)$. For references, we refer to \cite{Hum}.
\hip 

There is a formula due to Steinberg expressing the growth series $f_S(t)$ of a Coxeter system $(W,S)$ in terms of its finite special subgroups $W_T$ for $T\subseteq S\,,$
\begin{equation}\label{eq:Steinberg}
\frac{1}{f_S(t^ {-1})}=\sum\limits_{{W_T<W\atop\scriptscriptstyle{{\vert W_T\vert<\infty}}}}\,
\frac{(-1)^{\vert T\vert}}{f_T(t)}\,,
\end{equation}
where $W_{\varnothing}=\{1\}$.
By a result of Solomon, the growth polynomial of each term  $f_T(t)$ in \eqref{eq:Steinberg} can be expressed by means of its exponents $\{m_1,m_2,\ldots,m_p\}$
according to the formula
\begin{equation}\label{eq:Solomon}
f_T(t)=\prod\limits_{i=1}^ {p}\,[m_i+1]\,, 
\end{equation}
where $\,[k]=1+t+\dots+ t^{k-1}$ and, more generally, $\,[k_1,\dots,k_r]:=[k_1]\dots[k_r]\,$. 
A complete list of the irreducible spherical Coxeter groups together with their exponents can be found in \cite{Perren}.  For example, the exponents of the Coxeter group $\hbox{A}_{n}$ with Coxeter graph $\underbrace{
\begin{tikzpicture}
\fill[black] (0,0) circle (0.06cm);
\fill[black] (1/2,0) circle (0.06cm);
\fill[black] (9/6,0) circle (0.06cm);
\fill[black] (12/6,0) circle (0.06cm);
\draw (0,0) -- (1/2,0);
\draw (1/2,0) -- (2/3,0) ;
\draw (2/3,0) -- (4/3,0)[dotted];
\draw (4/3,0) -- (9/6,0) ;
\draw (9/6,0) -- (12/6,0) ;
\end{tikzpicture}}_{n} \,\,$ are $\{1,2,\dots,n\}$ so that 
\begin{equation}\label{An}
f_{A_n}(t)=[2,\dots,n+1]\,. 
\end{equation}
 
\hop
Furthermore, the growth series of a reducible Coxeter system $(W,S)$ with factor groups $(W_1,S_1)$ and $(W_2,S_2)$ such that
$S=(S_1\times\{1_{W_2}\})\cup(\{1_{W_1}\}\times S_2)$ satisfies the product formula
\[ f_S(t)=f_{S_1}(t)\cdot f_{S_1}(t). \]

\hip

In its disk of convergence, the growth series $f_S(t)$ is a rational function, which can be expressed as the quotient of coprime monic polynomials $p(t), q(t)\in\mathbb Z[t]$ of the same degree. 
The {\em growth rate} $\tau_{W}=\tau_{(W,S)}$ is defined by the inverse of the radius of convergence of $f_S(t)$ and can be expressed by 
\[
\tau_{W}=\limsup_{k\rightarrow\infty} {a_k}^{1/k}.
\]
It is the inverse of the smallest positive real pole of $f_S(t)$ and hence an algebraic integer. 

\hip

Important for the proof of our Theorem is the following result of Terragni \cite{Terragni2} about the growth monotonicity.

\begin{theorem} \label{terr} Let $(W,S)$ and $(W',S')$  be two Coxeter systems such that there is an injective map $ \iota:S\rightarrow S' $ with  $ m_{st} \le m'_{\iota(s)\iota(t)}$  for all $ s,t\in S$. 
 Then, $\tau_{(W,S)}\le \tau_{(W',S')}.$ 
\end{theorem}

For $n\geq 2$, consider a Coxeter group $\Gamma \subset \hbox{Isom}\mathbb H^n$ of finite covolume. 
By results of Milnor and de la Harpe, we know that $\tau_{\Gamma}>1$.  
More precisely, and as shown by Terragni \cite{Terragni2}, $ \tau_{\Gamma} \geq  \tau_{\Gamma_9} \approx 1.1380 $, where $\Gamma_9$ is the Coxeter simplex group given in Table \ref{Poly}.

\hip

Next, we introduce another tool in the proof of our result, the {\it extension} of a Coxeter graph.  

\begin{definition} Let $\Sigma$ be an abstract Coxeter graph.
An $\it extension$ of $\Sigma$ is a Coxeter graph $\Sigma'$ obtained by adding one node linked with a (simple) edge to the Coxeter graph $\Sigma$. 
\end{definition}\hip
As a direct consequence of Theorem \ref{terr}, if $W$ is a Coxeter group with Coxeter graph $\Sigma$, any extension $\Sigma'$  of $\Sigma$ encodes a Coxeter group $W'$ such that $\tau_{W}\leq \tau_{W'}$.

\begin{example}\label{affine3} Consider an irreducible affine Coxeter graph of order $3$ as given in Table \ref{affinediag}. Up to symmetry, the graph $\widetilde{A}_2$ has a unique extension given by the Coxeter graph at the top left in Figure \ref{G2etB2}. This graph describes the  Coxeter tetrahedron $[3,3^{[3]}]$ of finite volume. 
The Coxeter graphs $\widetilde C_2$ and $\widetilde G_2$  give rise to the remaining five extensions depicted in Figure \ref{G2etB2}.
By a result of Kellerhals \cite{K0}, these
six Coxeter graphs describe Coxeter tetrahedral groups $\Lambda$  of finite covolume in $\hbox{Isom}\mathbb H^3$ whose growth rates satisfy $\tau_{\Lambda}\geq \tau_{\Gamma_3}$.

\begin{figure}[!h]
\centering
 \begin{tikzpicture}\tikzstyle{every node}=[font=\small]
\fill[black] (0,0) circle (0.06cm);
\fill[black] (1/2,1/4) circle (0.06cm);
\fill[black] (1/2,-1/4) circle (0.06cm);
\fill[black] (-1/2,0) circle (0.06cm);
\draw (0,0) -- (1/2,1/4) ;
\draw (1/2,-1/4) -- (0,0)  ;
\draw (1/2,-1/4) -- (1/2,1/4)  ;
\draw (0,0) -- (-1/2,0) ;
\end{tikzpicture} 
 \qquad  \quad
\begin{tikzpicture}
\tikzstyle{every node}=[font=\small]
\fill[black] (0,0) circle (0.06cm);
\fill[black] (1/2,0) circle (0.06cm);
\fill[black] (1,0) circle (0.06cm);
\fill[black] (3/2,0) circle (0.06cm);
\draw (0,0) -- (1/2,0) node [above,midway] {$4$} ;
\draw (1,0) -- (1/2,0) node [above,midway] {$4$};
\draw (1,0) -- (3/2,0) ;
\end{tikzpicture}
 \qquad  \quad
\begin{tikzpicture}
\tikzstyle{every node}=[font=\small]
\fill[black] (0,0) circle (0.06cm);
\fill[black] (1/2,0) circle (0.06cm);
\fill[black] (2/2,0) circle (0.06cm);
\fill[black] (1/2,1/3) circle (0.06cm);
\draw (0,0) -- (1/2,0) node [above,midway] {$4$} ;
\draw (1,0) -- (1/2,0) node [above,midway] {$4$} ;
\draw (1/2,0) -- (1/2,1/3) ; 
\end{tikzpicture}

\bigskip
\begin{tikzpicture}
\tikzstyle{every node}=[font=\small]
\fill[black] (0,0) circle (0.06cm);
\fill[black] (1/2,0) circle (0.06cm);
\fill[black] (1,0) circle (0.06cm);
\fill[black] (3/2,0) circle (0.06cm);
\draw (0,0) -- (1/2,0) node [above,midway] {$6$} ;
\draw (1,0) -- (1/2,0) ;
\draw (1,0) -- (3/2,0) ;
\end{tikzpicture}  \qquad \quad
\begin{tikzpicture}
\tikzstyle{every node}=[font=\small]
\fill[black] (0,0) circle (0.06cm);
\fill[black] (1/2,0) circle (0.06cm);
\fill[black] (2/2,0) circle (0.06cm);
\fill[black] (3/2,0) circle (0.06cm);
\draw (0,0) -- (1/2,0) ;
\draw (1,0) -- (1/2,0) node [above,midway] {$6$} ;
\draw (1,0) -- (3/2,0) ;
\end{tikzpicture} \qquad  \quad
\begin{tikzpicture}
\tikzstyle{every node}=[font=\small]
\fill[black] (0,0) circle (0.06cm);
\fill[black] (1/2,0) circle (0.06cm);
\fill[black] (2/2,0) circle (0.06cm);
\fill[black] (1/2,1/3) circle (0.06cm);
\draw (0,0) -- (1/2,0) node [above,midway] {$6$} ;
\draw (1,0) -- (1/2,0) ;
\draw (1/2,0) -- (1/2,1/3) ;
\end{tikzpicture}  
\bigskip
\caption{Extensions of $\widetilde A_2$, $\widetilde C_2$ and $\widetilde G_2$}
\label{G2etB2}
\end{figure}
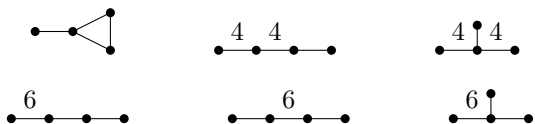
\end{example}

\begin{example}\label{affine4}
In a similar way, any extension of an irreducible affine Coxeter graph of order $4$ yields a Coxeter simplex group of finite covolume in $\hbox{Isom}\mathbb H^4$.  
They are given in Figure \ref{B3etC3}. Notice that $\Gamma_4=[4,3^{2,1}]$ is part of them.

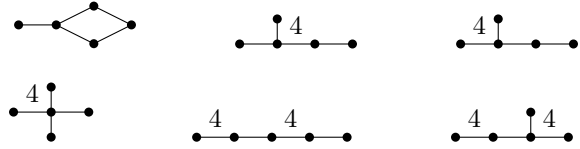
\begin{figure}[!h]
\centering
\begin{tikzpicture}
\fill[black] (0,0) circle (0.06cm);
\fill[black] (1/2,1/4) circle (0.06cm);
\fill[black] (1/2,-1/4) circle (0.06cm);
\fill[black] (-1/2,0) circle (0.06cm);
\fill[black] (1,0) circle (0.06cm);
\draw (0,0) -- (1/2,1/4) ;
\draw (1/2,-1/4) -- (0,0)  ;
\draw (1/2,-1/4) -- (1,0)  ;
\draw (1/2,1/4) -- (1,0)  ;
\draw (0,0) -- (-1/2,0) ;
\end{tikzpicture} 
 \qquad \quad 
\begin{tikzpicture} 
\tikzstyle{every node}=[font=\small]
\fill[black] (0,0) circle (0.06cm); \fill[black] (1/2,0) circle (0.06cm); \fill[black] (2/2,0) circle (0.06cm); \fill[black] (3/2,0) circle (0.06cm); \fill[black] (1/2,1/3) circle (0.06cm); \draw (0,0) -- (1/2,0)   ; \draw (1/2,0) -- (2/2,0)node [above,midway] {$4$} ; \draw (2/2,0) -- (3/2,0)  ; \draw (1/2,1/3) -- (1/2,0); \end{tikzpicture}     
  \qquad \quad  \begin{tikzpicture}
\tikzstyle{every node}=[font=\small]
\fill[black] (0,0) circle (0.06cm);
\fill[black] (1/2,0) circle (0.06cm);
\fill[black] (2/2,0) circle (0.06cm);
\fill[black] (3/2,0) circle (0.06cm);
\fill[black] (1/2,1/3) circle (0.06cm);
\draw (0,0) -- (1/2,0)  node [above,midway] {$4$} ;
\draw (1/2,0) -- (2/2,0) ;
\draw (2/2,0) -- (3/2,0)  ;
\draw (1/2,1/3) -- (1/2,0);
\end{tikzpicture}  

\bigskip
\begin{tikzpicture}
\tikzstyle{every node}=[font=\small]
\fill[black] (0,0) circle (0.06cm);
\fill[black] (1/2,0) circle (0.06cm);
\fill[black] (1,0) circle (0.06cm);
\fill[black] (1/2,1/3) circle (0.06cm);
\fill[black] (1/2,-1/3) circle (0.06cm);
\draw (0,0) -- (1/2,0)  node [above,midway] {$4$} ;
\draw (1,0) -- (1/2,0) ;
\draw (1/2,0) -- (1/2,1/3) ;
\draw (1/2,0) -- (1/2,-1/3) ;
\end{tikzpicture} 
\qquad \quad 
\begin{tikzpicture}
\tikzstyle{every node}=[font=\small]
\fill[black] (0,0) circle (0.06cm);
\fill[black] (1/2,0) circle (0.06cm);
\fill[black] (1,0) circle (0.06cm);
\fill[black] (3/2,0) circle (0.06cm);
\fill[black] (2,0) circle (0.06cm);
\draw (0,0) -- (1/2,0) node [above,midway] {$4$} ;
\draw (1,0) -- (1/2,0) ;
\draw (1.5,0) -- (2,0) ;
\draw (1,0) -- (3/2,0) node [above,midway] {$4$};
\end{tikzpicture}
 \qquad \quad 
\begin{tikzpicture}
\tikzstyle{every node}=[font=\small]
\fill[black] (0,0) circle (0.06cm);
\fill[black] (1/2,0) circle (0.06cm);
\fill[black] (1,0) circle (0.06cm);
\fill[black] (3/2,0) circle (0.06cm);
\fill[black] (1,1/3) circle (0.06cm);
\draw (0,0) -- (1/2,0) node [above,midway] {$4$} ;
\draw (1,0) -- (1/2,0) ;
\draw (1,0) -- (1,1/3) ;
\draw (1,0) -- (3/2,0) node [above,midway] {$4$};
\end{tikzpicture}
\bigskip
\caption{Extensions of $\widetilde A_3$, $\widetilde B_3$ and $\widetilde C_3$}
\label{B3etC3}
\end{figure}
\end{example}

\begin{remark}\label{affinerank}  
When considering irreducible affine Coxeter graphs of order greater than or equal to $5$, the resulting extensions do {\em not} always relate to hyperbolic Coxeter $n$-simplex groups of finite covolume. 
For example, among the extensions of $\widetilde {F_4}$, the graph depicted in Figure \ref{finiteandnot} describes an infinite volume Coxeter simplex in $\mathbb H^5$.
\end{remark}

\begin{figure}[!h]
\begin{center}
\begin{tikzpicture}
\tikzstyle{every node}=[font=\small]
\fill[black] (0,0) circle (0.06cm);
\fill[black] (1/2,0) circle (0.06cm);
\fill[black] (2/2,0) circle (0.06cm);
\fill[black] (3/2,0) circle (0.06cm);
\fill[black] (4/2,0) circle (0.06cm);
\fill[black] (1/2,1/3) circle (0.06cm);
\draw (0,0) -- (1/2,0)  ; 
\draw (1/2,1/3) -- (1/2,0)  ; 
\draw (2/2,0) -- (1/2,0) node [below, midway]{$4$}; 
\draw (2/2,0) -- (3/2,0) ;
\draw (3/2,0) -- (4/2,0) ; 
\end{tikzpicture}
\end{center}
\caption{An infinite volume Coxeter $5$-simplex}
\label{finiteandnot}
\end{figure}
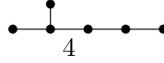
\hop
\hop

\section{Proof of the Theorem}\label{section3}

Let $2 \leq n\leq 9$, and consider the Coxeter simplex group $\Gamma_n \subset \hbox{Isom}\mathbb H^n$ whose Coxeter graph is depicted in Table \ref{Poly}. In this section, we provide the proof of our main result stated as follows.

\begin{thm*}
For any $2 \leq n\leq 9$, the group $\Gamma_{n} $ has minimal growth rate among all non-cocompact hyperbolic Coxeter groups of finite covolume in $\hbox{Isom}\mathbb H^n$, and as such it is unique.
\end{thm*}

\hip
 For $n=2$ and for $n=3$, the result has been established by Floyd \cite{Fl} and Kellerhals \cite{K0}. 
Therefore it suffices to prove the Theorem for $4\le n \le 9$.

\hip
Observe that the growth rates of all Coxeter simplex groups in  $\hbox{Isom}\mathbb H^n$ are known. Their list can be found in \cite{Terragni2}.
 In particular, one deduces the following strict inequalities.
\begin{equation}\label{decroissancetau}
\tau_{\Gamma_9} \approx   1.1380 < \dots < \tau_{\Gamma_5} \approx 1.2481 < \tau_{\Gamma_4} \approx 1.3717 \,. 
\end{equation}

\vskip-.5cm
\begin{equation}\label{eq633}
\tau_{\Gamma_5}  < \tau_{\Gamma_3}\approx 1.2964 \, .
\end{equation}

For fixed dimension $n$, one also checks that $\Gamma_n$ has minimal growth rate among (all the finitely many) Coxeter simplex groups $\Lambda \subset \hbox{Isom}\mathbb H^n$.  

As a consequence, we focus on hyperbolic Coxeter groups $\Gamma \subset \hbox{Isom}\mathbb H^n$  generated by at least $N \geq n+2$ reflections in the facets of a non-compact finite volume Coxeter polyhedron $P\subset\mathbb H^n$. We have to show that $\tau_{\Gamma_n}<\tau_{\Gamma}$. 

Suppose that the Coxeter polyhedron $P$ is simple. 
By Theorem $\ref{thmFT}$, $P$ is either isometric to the polyhedron $P_0\subset \hbox{Isom}\mathbb H^4$ depicted in Figure \ref{fig:unique}, or $P$ has a pair of disjoint facets. For the growth rate $\tau$ of the Coxeter group associated to $P_0$, one easily checks with help of the software CoxIter \cite{Gug1, Gug2}
that $\tau_{\Gamma_4} <\tau \approx 2.8383 $. Hence, we can assume that $P$ is not isometric to $P_0$.
If $P$ has a pair of disjoint facets, then the Coxeter graph $\Sigma$ of $P$ and its associated group $\Gamma$ contains a subgraph \begin{tikzpicture}
\fill[black] (0,0) circle (0.06cm);
\fill[black] (1/2,0) circle (0.06cm);
\draw (0,0) -- (1/2,0) node [above, midway]{$\infty$}; 
\end{tikzpicture}\,. 
\\The property that the Coxeter graph $\Sigma$ contains such a subgraph of type $\widetilde A_1 = [\infty]$ allows us to conclude the proof, whether the polyhedron $P$ is simple or not. In the following, we first look at this property and analyse it more closely.

\hop

\subsection{In the presence of $\widetilde A_1$}\label{section3.1}

We start by considering particular Coxeter graphs of order $4$ containing $\widetilde A_1$. Their related growth rates will be useful when comparing with the one of $\Gamma$. This approach is similar to the one developed in \cite{Bred-Kell}. 
\hop
\hop

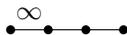
\begin{figure}[!h]
\centering
\begin{tikzpicture}
\tikzstyle{every node}=[font=\small]
\fill[black] (0,0) circle (0.06cm);
\fill[black] (1/2,0) circle (0.06cm);
\fill[black] (2/2,0) circle (0.06cm);
\fill[black] (3/2,0) circle (0.06cm);
\draw (0,0) -- (1/2,0) node [above, midway]{$\infty$};
\draw (1/2,0) -- (2/2,0) ;
\draw (3/2,0) -- (2/2,0) ;
\end{tikzpicture} 
\caption{The Coxeter group $W_0=[\infty,3,3]$}
\label{Sigma0}
\end{figure}

Let $W_0 = [\infty,3,3]$ be the abstract Coxeter group
depicted in Figure \ref{Sigma0}.
By means of the software CoxIter, one checks that 
\begin{equation}\label{infty}
\tau_{\Gamma_4}< \tau_{W_0} \approx 1.4655 \, .
\end{equation}

\hip

Furthermore, consider the two abstract Coxeter groups $W_1=[3,\infty,3]$ and $W_2=[\infty, 3^{1,1}]$ given in Figure \ref{W1W2}. 

\begin{figure}[!h]
\centering
 \begin{tikzpicture}
\tikzstyle{every node}=[font=\small]
\fill[black] (0,0) circle  (0.06cm);
\fill[black] (1/2,0) circle (0.06cm);
\fill[black] (2/2,0) circle (0.06cm);
\fill[black] (3/2,0) circle (0.06cm);
\draw (0,0) -- (1/2,0) ; 
\draw (1/2,0) -- (2/2,0) node [above, midway]{$\infty$}; 
\draw (2/2,0) -- (3/2,0) ;
\end{tikzpicture}  
\qquad  \qquad 
\begin{tikzpicture}
\tikzstyle{every node}=[font=\small]
\fill[black] (0,0) circle (0.06cm);
\fill[black] (1/2,0) circle (0.06cm);
\fill[black] (2/2,0) circle (0.06cm);
\fill[black] (1/2,1/3) circle (0.06cm);
\draw (0,0) -- (1/2,0) node [above, midway]{$\infty$}; 
\draw (2/2,-0) -- (1/2,0) ;
\draw (1/2,0) -- (1/2,1/3) ; 
\end{tikzpicture}
\caption{The Coxeter groups $W_1=[3,\infty,3]$ and $W_2=[\infty, 3^{1,1}]$}
\label{W1W2}
\end{figure}
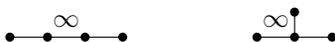

For their growth rates, we prove the following auxiliary result.  

\begin{lemma} \label{complemma}
$\tau_{W_0} < \tau_{W_1}$ and $ \tau_{W_0}< \tau_{W_2}$.
\end{lemma}

\begin{proof}

For $0\leq i\leq 2$, denote by $f_i:=f_{W_i}$ the growth series of $W_i$ and  by $R_i$  its radius of convergence. Recall that $R_i$ is the smallest positive pole of $f_i$, and that $\tau_{W_i}=\frac1{R_i}$. 
\hip
We establish the growth functions $f_i$ according to Steinberg's formula (\ref{eq:Steinberg}). They are given as follows. 

$$
\begin{array}{lll}
\frac{1}{f_{0}(t^{-1})}=& 1-\frac{4}{[2]}+\frac3{[2,2]}+\frac2{[2,3]}- \frac1{[2,2,3]}-\frac1{[2,3,4]}  \quad ; \\
\\
\frac{1}{f_1(t^{-1})}= &1-\frac{4}{[2]}+\frac3{[2,2]} +\frac2{[2,3]}- \frac2{[2,2,3]} \quad  ; \\
\\
\frac{1}{f_2(t^{-1})}=&  1-\frac{4}{[2]}+\frac3{[2,2]}+\frac2{[2,3]}- \frac1{[2,2,2]}-\frac1{[2,3,4]}  \quad.
\end{array}
$$

\hop
Hence, for any $t>0$, one has the positive difference functions given by
$$
\begin{array}{lll}
\frac1{f_{0}(t^{-1})}-\frac{1}{f_1(t^{-1})}&=\frac1{[2,2,3]} - \frac1{[2,3,4]} =\frac{t^2+t^3}{[2,2,3,4]} > 0 \quad; \\
\\
\frac1{f_{0}(t^{-1})}-\frac{1}{f_2(t^{-1})}&=\frac1{[2,2,2]} - \frac1{[2,2,3]} =\frac{t^2}{[2,2,2,3]} > 0 \quad.
\end{array}
$$
\hip
Therefore, for $i=1,2$, and for $u=t^{-1}\in (0,1)$, the smallest positive root $R_0$ of $\frac1{f_{0}(u)}$ is strictly bigger than the one of $\frac1{f_{i}(u)}$. This finishes the proof.

\end{proof}

\hop
As a first consequence, combining (\ref{decroissancetau}), (\ref{infty}) and Lemma \ref{complemma}, one obtains that \begin{equation}\label{Sigmai}
\tau_{\Gamma_n} <\tau_{W_i} \quad
\end{equation} 
 for all $4\leq n \leq 9$ and  $0\le i \le 2$. 

\hip 

Next, suppose that the Coxeter graph $\Sigma$ of $\Gamma$ contains a subgraph $\widetilde A_1$. 
Since $\Sigma$ is connected of order $N \geq  n+2 \geq 6$, the subgraph $\widetilde A_1$ is contained in a connected subgraph $\sigma$ of order $4$ in $\Sigma$,
which is related to a special subgroup $W$ of $\Gamma$. %
By Theorem \ref{terr}, 
 one has that $\tau_{W_i}\leq \tau_{W}$ for some $0\le i \le 2$. By combining  (\ref{Sigmai}) with these findings, and by Theorem \ref{terr} and Lemma \ref{complemma}, one deduces that 
\begin{equation}\label{simple}
\tau_{\Gamma_n} < \tau_{W_0} \le \tau_{W} \le \tau_{\Gamma} \, .
\end{equation}

This finishes the proof of the Theorem in the presence of a subgraph $\widetilde A_1$ in $\Sigma$.

\subsection{In the absence of $\widetilde A_1$}\label{section3.2}

Suppose that the Coxeter graph $\Sigma$ with $N\geq n+2$ nodes does {\it not} contain a subgraph of type $\widetilde A_1$. In particular, by Theorem \ref{thmFT}, the corresponding Coxeter polyhedron $P \subset \hbox{Isom}\mathbb H^n$ is not simple, and it follows that $5\le n\le 9$. 

Consider a non-simple ideal vertex $v_{\infty}\in P$. Its link $L_{\infty} \subset \mathbb E^{n-1}$ is described by a reducible affine subgraph $\sigma_{\infty}$ with $n_c=n_c(\infty)\geq 2$ components which satisfies $
n-1=\hbox{order}(\sigma_{\infty})-n_{c}$ by (\ref{rankorder}).
In Table \ref{2components}, we list all possible realisations for 
$\sigma_{\infty}$ by using the following notations.\\
Let $\widetilde\sigma_{k}$ be a connected affine Coxeter graph of order $k\geq 3$ as listed in Table \ref{affinediag}, and denote by  $\bigsqcup\limits_k \widetilde\sigma_{k}$ the Coxeter graph consisting of the components of type $\widetilde\sigma_{k}$.

\begin{table}[!h]
\caption{Reducible affine Coxeter graphs $\sigma_{\infty}$ with $n_c\geq 2$ components $\widetilde\sigma_{k}$ of order $k\geq 3$ such that $n=\hbox{order}(\sigma_{\infty})-n_{c}+1$}
\renewcommand*{\arraystretch}{1.3}
\begin{center}
\begin{tabular}{|c| c |c|c|c|c |}
\hline
$n$ & $5$& $6$&$7$&$8$ & $9$\\
\hline
\multicolumn{1}{c|}{} &$\widetilde\sigma_3 \sqcup \widetilde\sigma_3$& $\widetilde\sigma_3 \sqcup \widetilde\sigma_4$  & $\widetilde\sigma_3 \sqcup \widetilde\sigma_5$ & $\widetilde\sigma_3 \sqcup \widetilde\sigma_6$ & $\widetilde\sigma_3 \sqcup \widetilde\sigma_7$\\
\multicolumn{3}{c|}{}  & $\widetilde\sigma_4 \sqcup \widetilde\sigma_4$ & $\widetilde\sigma_4 \sqcup \widetilde\sigma_5$ & $\widetilde\sigma_4 \sqcup \widetilde\sigma_6$ \\
\multicolumn{3}{c|}{}   &  $\widetilde\sigma_3 \sqcup \widetilde\sigma_3 \sqcup \widetilde\sigma_3$& $\widetilde\sigma_3 \sqcup \widetilde\sigma_3 \sqcup \widetilde\sigma_4$ & $\widetilde\sigma_5 \sqcup \widetilde\sigma_5$\\
\multicolumn{5}{c|}{}    & $\widetilde\sigma_3 \sqcup \widetilde\sigma_4 \sqcup \widetilde\sigma_4$ \\
\multicolumn{5}{c|}{}   & $\widetilde\sigma_3 \sqcup \widetilde\sigma_3 \sqcup \widetilde\sigma_3 \sqcup \widetilde\sigma_3$ \\
\end{tabular}
\end{center}
\label{2components}
\end{table}

\hip

Observe that for any graph $\bigsqcup\limits_{k} \widetilde \sigma_{k}$
 in Table \ref{2components}, one has $3\leq \min\limits_k k \leq 5$,  and that the case $\min\limits_k k=5$ appears only when $n=9$. \hip Among the different components of $\sigma_{\infty}$, we consider the ones of smallest order $\geq 3$ together with their extensions. 
\hop

\noindent
$\bullet \quad$ Assume that the graph $\sigma_{\infty}$ of the vertex link $L_{\infty}$ contains an affine component $\widetilde \sigma$ of order $3$.
By Example $\ref{affine3}$, we know that any extension of $\widetilde \sigma$ encodes a Coxeter tetrahedral group $\Lambda\subset\hbox{Isom}\mathbb H^3$ of finite covolume. The graph $\Sigma$ itself contains a subgraph $\sigma$ of order $4$ which in turn comprises $\widetilde \sigma$. The Coxeter graph $\sigma$ corresponds to a special subgroup $W$ of $\Gamma$, and by Theorem \ref{terr}, we deduce that $\tau_{\Lambda}\leq \tau_{W}$.

Since $\tau_{\Gamma_3}\le \tau_{\Lambda}$, and in view of  (\ref{decroissancetau}) and (\ref{eq633}), Theorem \ref{terr} yields the desired inequality  \begin{equation}\label{comp3}
\tau_{\Gamma_{n}} < \tau_{\Gamma_3} \leq  \tau_{\Lambda} \leq \tau_W \leq \tau_{\Gamma} \,,
\end{equation}
which finishes the proof in this case, and for $n=5$ and $n=6$; see Table \ref{2components}. 

\hop

\noindent
$\bullet \quad$ Assume that the graph $\sigma_{\infty}$ contains an affine component $\widetilde \sigma$ of order $4$. We apply the same reasoning as above. By Example $\ref{affine4}$, any extension of $\widetilde \sigma$ corresponds to a Coxeter 4-simplex group $\Lambda$ of finite covolume, and $\tau_{\Gamma_4} \leq \tau_{\Lambda}$. Again, $\Sigma$ contains a subgraph $\sigma$ comprising $\widetilde \sigma$. Hence, there exists a special subgroup $W$ of $\Gamma$ described by $\sigma$ so that
\begin{equation}\label{comp4}
\tau_{\Gamma_n}  < \tau_{\Gamma_4} \leq \tau_{\Lambda} \leq \tau_{W} \leq \tau_{\Gamma}   \,.
\end{equation}

\hip

By
(\ref{comp3}) and (\ref{comp4}) the proof is finished in this case, and for $n=7$ and $n=8$; see Table \ref{2components}.

\hip

\noindent
$\bullet \quad$ Assume that $\sigma_{\infty}$ contains an affine component $\widetilde \sigma$ of order $5$.  By Table \ref{2components}, one has $7\leq n\leq 9$. It is not difficult to list all possible extensions of $\widetilde \sigma$. There are exactly fifteen such extensions. It turns out that there are eleven extensions that encode Coxeter $5$-simplex groups of finite covolume, while the remaining four extensions describe Coxeter $5$-simplex groups $\Delta_i, i=1,\dots,4,$ of {\it infinite} covolume. These last four simplices arise by extending  $\widetilde B_4$, $\widetilde C_4$ and $\widetilde F_4$. They are given in Figure \ref{nonfinite}, together with their associated growth rates computed with CoxIter.

\hop
\begin{figure}[!h]
\begin{center}
\begin{tikzpicture}
\tikzstyle{every node}=[font=\small]
\fill[black] (0,0) circle (0.06cm);
\fill[black] (1/2,0) circle (0.06cm);
\fill[black] (2/2,0) circle (0.06cm);
\fill[black] (3/2,0) circle (0.06cm);
\fill[black] (2/2,1/3) circle (0.06cm);
\fill[black] (1/2,1/3) circle (0.06cm);
\draw (0,0) -- (1/2,0)   node [below,midway] {4};
\draw (1/2,0) -- (2/2,0);
\draw (1/2,1/3) -- (1/2,0) ;
\draw (2/2,0) -- (3/2,0) ;
\draw (2/2,1/3) -- (2/2,0)  ;
\end{tikzpicture}
\quad  \quad
\begin{tikzpicture}
\tikzstyle{every node}=[font=\small]
\fill[black] (0,0) circle (0.06cm);
\fill[black] (1/2,0) circle (0.06cm);
\fill[black] (1,0) circle (0.06cm);
\fill[black] (3/2,0) circle (0.06cm);
\fill[black] (2,0) circle (0.06cm);
\fill[black] (1/2,1/3) circle (0.06cm);
\draw (0,0) -- (1/2,0)  node [below, midway]{$4$} ;
\draw (1,0) -- (1/2,0) ;
\draw (3/2,0) -- (1,0);
\draw (1/2,0) -- (1/2,1/3);
\draw (2,0) -- (3/2,0) node [below,midway] {4};
\end{tikzpicture}
\quad \quad
\begin{tikzpicture}
\tikzstyle{every node}=[font=\small]
\fill[black] (0,0) circle (0.06cm);
\fill[black] (1/2,0) circle (0.06cm);
\fill[black] (1,0) circle (0.06cm);
\fill[black] (3/2,0) circle (0.06cm);
\fill[black] (2,0) circle (0.06cm);
\fill[black] (1/2,1/3) circle (0.06cm);
\draw (0,0) -- (1/2,0) ;
\draw (1,0) -- (1/2,0)  node [below, midway]{$4$} ;
\draw (3/2,0) -- (1,0);
\draw (1/2,0) -- (1/2,1/3);
\draw (2,0) -- (3/2,0);
\end{tikzpicture}
\quad \quad
\begin{tikzpicture}
\tikzstyle{every node}=[font=\small]
\fill[black] (0,0) circle (0.06cm);
\fill[black] (1/2,0) circle (0.06cm);
\fill[black] (1,0) circle (0.06cm);
\fill[black] (3/2,0) circle (0.06cm);
\fill[black] (2,0) circle (0.06cm);
\fill[black] (1,1/3) circle (0.06cm);
\draw (0,0) -- (1/2,0)  ;
\draw (1,0) -- (1/2,0) node [below, midway]{$4$} ;
\draw (3/2,0) -- (1,0);
\draw (1,1/3) -- (1,0);
\draw (2,0) -- (3/2,0);
\end{tikzpicture}
\end{center}
\begin{center}
$\tau_{\Delta_1}\approx 1.678 \quad \qquad 
\tau_{\Delta_2}\approx 1.599 \quad \qquad 
\tau_{\Delta_3}\approx 1.668 \quad \qquad 
\tau_{\Delta_4}\approx 1.702 
$
\end{center}
\caption{The Coxeter groups $\Delta_i, i=1,\dots,4$}
\label{nonfinite}
\end{figure}
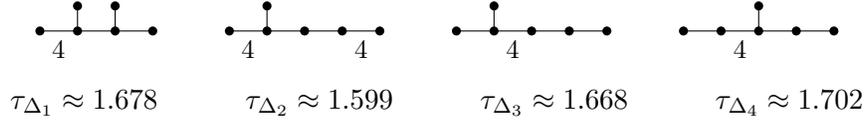

\hip In view of (\ref{decroissancetau}), it turns out that \begin{equation}\label{Ei}
\tau_{\Gamma_5}<\tau_{\Delta_i} \quad \hbox{ for } i=1,\dots,4 \, .
\end{equation}
\hip
As above, the component $\widetilde \sigma$ lies in a subgraph $\sigma$ of order $6$ in $\Sigma$, and the latter corresponds to a special subgroup $W$ of $\Gamma$ so that \[
\hbox{ either } \quad \tau_{\Lambda}\leq \tau_{W}   \quad\hbox{ or } \quad \tau_{\Delta_i}\leq \tau_{W} \quad, \quad  1\leq i\leq 4\ ,
\]
where $\Lambda$ is a Coxeter $5$-simplex group of finite covolume. 
Since $ \tau_{ \Gamma_5} \leq \tau_{\Lambda}$, and by  (\ref{decroissancetau}), (\ref{Ei}), one deduces that
 \begin{equation}\label{comp5}
 \tau_{\Gamma_n} < \tau_{ \Gamma_5} \leq \tau_{W} \leq \tau_{\Gamma}\, .
\end{equation}
 
\hop

This finishes the proof of this case. 

\noindent
Finally, all the above considerations allow us to
conclude the proof of the Theorem.

{\hfill$\square$}


\begin{thebibliography}{9}

\bibitem{Bou}
N. Bourbaki, 
{\em Groupes et algebres de Lie}
Ch. 4-6. Hermann, Paris (1968).


\bibitem{Bred-Kell}
N. Bredon,  R. Kellerhals, 
{\em {Hyperbolic Coxeter groups and minimal growth rates in dimensions four and five}},
arXiv:2008.10961.v3, to appear in Groups, Geometry and Dynamics. 


\bibitem{Coxeter}
H. S. M. Coxeter,  
{\em Discrete groups generated by reflections}, Ann. Math. 35 (1934),  588--621.


\bibitem{F-web}
A. Felikson,
{\em Hyperbolic  Coxeter polytopes},
\url{https://www.maths.dur.ac.uk/users/anna.felikson/Polytopes/polytopes.html}

\bibitem{FT1} 
A. Felikson,  P. Tumarkin, 
{\em On hyperbolic Coxeter polytopes with mutually intersecting facets}, 
J. Combin. Theory Ser. A 115 (2008), 121--146.

\bibitem{Fl}
W. Floyd, 
{\em Growth of planar {C}oxeter groups}, 
PV numbers, and {S}alem numbers, Math. Ann. 293 (1992) 475--483.

\bibitem{Gug1}
R. Guglielmetti, 
{\em CoxIter -- computing invariants of hyperbolic Coxeter groups},
LMS J. Comput. Math. 18 (2015), 754--773.

\bibitem{Gug2}
R. Guglielmetti, 
{\em CoxIterWeb},\\
\url{https://coxiterweb.rafaelguglielmetti.ch/}

\bibitem{Hild}
T. Hild,
{\em The cusped hyperbolic orbifolds of minimal volume in dimensions less than ten}, 
J. Algebra 313 (2007), 208--222. 

\bibitem{HK}
T. Hild, R. Kellerhals, 
{\em The fcc lattice and the cusped hyperbolic 4-orbifold of minimal volume: In memoriam H. S. M. Coxeter},
J. Lond. Math. Soc. (2) 75 (2007), 677–689.


\bibitem{Hum}
J. Humphreys, 
{\em Reflection groups and Coxeter groups}, 
Cambridge Studies in Advanced Mathematics, vol. 29, Cambridge University Press, Cambridge, 1990.

\bibitem{JKRT} 
N. Johnson, R. Kellerhals, J. Ratcliffe, and S. Tschantz, 
{\em The size of a hyperbolic Coxeter simplex}, 
Transform. Groups 4 (1999), 329--353.

\bibitem{K0}
R. Kellerhals,
{\em Cofinite hyperbolic Coxeter groups, minimal growth rate and Pisot numbers},
Algebr. Geom. Topol. 13 (2013), 1001--1025.


\bibitem{Mey}
R. Meyerhoff,
{\em The cusped hyperbolic 3-orbifold of minimum volume}
Bull. Amer. Math. Soc, 13 (1985), 154-156.

\bibitem{Perren}
G. Perren,
{\em Growth of cocompact hyperbolic Coxeter groups and their rate},
PhD thesis no. 1656, University of Fribourg, 2007.


\bibitem{Sie}
C. L. Siegel, 
{\em Some remarks on discontinuous groups}
Ann. Math.46 (1945), 708-718.

\bibitem{Terragni}
T. Terragni,
{\em On the growth of a Coxeter group},
Groups Geom. Dyn. 10 (2016), 601--618.

\bibitem{Terragni2}
T. Terragni,
{\em On the growth of a Coxeter group (extended version)},
arXiv:1312.3437v2, December 2013.


\bibitem{V1} 
{\` E}. Vinberg, 
{\em Hyperbolic reflection groups}, 
Uspekhi Mat. Nauk 40 (1985), 29--66, 255.


\bibitem{VinGII}
{\` E}. Vinberg, 
{\em Geometry II},
Encyclopaedia of Mathematical Sciences, vol. 29. Springer-Verlag, Berlin, 1993.

\end{thebibliography}
\end{document}